\documentclass[12pt]{amsart}
\usepackage[utf8]{inputenc}
\usepackage{amsmath,amsfonts}
\usepackage{amsthm}
\usepackage{amssymb}
\usepackage{hyperref}
\usepackage[shortlabels]{enumitem}
\usepackage{bm}
\usepackage{multirow}
\usepackage{tikz-cd}
\newtheorem{thm}{Theorem}[section]
\newtheorem{lem}[thm]{Lemma}
\newtheorem{prop}[thm]{Proposition}
\newtheorem{cor}[thm]{Corollary}
\newtheorem*{thm*}{Theorem}
\newtheorem{conj}{Conjecture}
\theoremstyle{definition}
\newtheorem{definition}[thm]{Definition}
\newtheorem{remark}[thm]{Remark}
\newtheorem*{dfn*}{Definition}

\newtheorem*{ques*}{Question}

\newcommand{\R}{\mathbb{R}}
\newcommand{\N}{\mathbb{N}}
\newcommand{\Q}{\mathbb{Q}}

\newcommand{\Z}{\mathbb{Z}}

\title[The AMU conjecture for the derived subgroup of $J_2(\Sigma_{g,n})$]{The Andersen-Masbaum-Ueno conjecture for the derived subgroup of the Johnson kernel}
\author{Renaud Detcherry}
\address{Institut de Mathématiques de Bourgogne \& Institut Universitaire de France, UMR 5584 CNRS, Université Bourgogne Franche-Comté, F-2100 Dijon, France}
\email{renaud.detcherry@u-bourgogne.fr}

\date{}
\begin{document}
	
	\maketitle
	
	\begin{abstract}A conjecture of Andersen, Masbaum and Ueno states that for any compact oriented surface $\Sigma_{g,n}$ and any pseudo-Anosov $f\in \mathrm{Mod}(\Sigma_{g,n}),$ the matrix $\rho_r(f)$ has infinite order for any large $r,$ where $\rho_r$ is the $\mathrm{SO}(3)$-WRT quantum representation of the mapping class group $\mathrm{Mod}(\Sigma_{g,n})$ at a primitive $r$-th root of unity.
	We prove this conjecture for prime $r$ and any $f\in [J_2(\Sigma_{g,n}),J_2(\Sigma_{g,n})],$ where $J_2(\Sigma_{g,n})$ is the Johnson kernel.
	\end{abstract}

\section{Introduction}
\label{sec:intro}

Let $\Sigma_{g,n}$ be a compact oriented surface of genus $g$ and $n$ boundary components. The (pure) mapping class group of $\Sigma_{g,n}$ is
$$\mathrm{Mod}(\Sigma_{g,n})=\mathrm{Homeo}^+(\Sigma_{g,n})/\mathrm{Homeo}_0(\Sigma_{g,n})$$
where $\mathrm{Homeo}_0(\Sigma_{g,n})$ stands for homeomorphisms that are isotopic to the identity, and we require homeomorphisms and isotopy to be the identity on the boundary.

The Nielsen-Thurston classification separates mapping classes into $3$ large types: finite order elements, reducible elements and pseudo-Anosov elements. Pseudo-Anosov elements stabilize two transverse singular foliations $\mathcal{F}^u,\mathcal{F}^s$  on $\Sigma_{g,n},$ contracting a transverse measure on $\mathcal{F}^s$ and dilating a transverse measure on $\mathcal{F}^u$ by a factor $\lambda>1.$ They may also be characterized by their action on isotopy classes of simple closed curves having no finite orbit. 

Reducible maps stabilize (up to isotopy) a multicurve on $\Sigma_{g,n}.$ Up to taking a power, a reducible mapping class may be restricted to a partition of $\Sigma_{g,n}$ into subsurfaces; we then say that it has a pseudo-Anosov part if its restriction to one subsurface is pseudo-Anosov.

Mapping class groups of surfaces enjoy a rich variety of linear representations, and a natural question is whether pseudo-Anosov maps may be described in terms of representations of the mapping class groups. To this end, a conjecture of Andersen, Masbaum and Ueno gives a simple criterion for a mapping class to have a pseudo-Anosov part, in terms of the so-called $\mathrm{SO}_3$-Witten-Reshetikhin-Turaev representations. 

Let $S=\Sigma_{g,n}$ and $|\partial S|$ be the set of connected components of $\partial S.$ For any odd integer $r\geq 3,$ and any map
$$c:|\partial S|\longrightarrow I_r=\lbrace 0,2,\ldots,r-3\rbrace,$$
the $\mathrm{SO}_3$-WRT quantum representation $\rho_{r,c}$ is a map:
$$\rho_{r,c}: \mathrm{Mod}(\Sigma_{g,n})\longrightarrow \mathrm{PAut}(RT_r(S,c)),$$
where $RT_p(S,c)$ is a $\Q(\zeta_r)$-vector space, where $\zeta_r$ is a primitive $r$-th root of unity, called the WRT TQFT-vector space. When $S$ has boundary, we denote by $\rho_r$ the direct sum of the representations $\rho_{r,c}.$ The Andersen-Masbaum-Ueno conjecture (or AMU conjecture) is then the following:
\begin{conj}(AMU conjecture \cite{AMU06}) Let $\Sigma_{g,n}$ be a compact oriented surface and $f\in \mathrm{Mod}(\Sigma_{g,n}).$ Then $f$ has a pseudo-Anosov part if and only if $\rho_r(f)$ has infinite order for all large enough odd $r\geq 3.$ 
\end{conj}

It is not difficult to prove that when $f$ has no pseudo-Anosov part, then $\rho_r(f)$ has finite order; this observation is what motivated the conjecture in \cite{AMU06}. Moreover, it is not hard to show that the conjecture would follow from the case of pseudo-Anosov mapping classes.

The AMU conjecture has only been proved for a handful cases of surfaces $S$ and pseudo-Anosov maps $f.$ In \cite{AMU06} and \cite{San12}, the conjecture is proved for pseudo-Anosov mapping classes of $\mathrm{Mod}(\Sigma_{0,4})$ and $\mathrm{Mod}(\Sigma_{1,1})$ respectively, by relating the quantum representation to the homology representation of the mapping class group. In \cite{EgJo} and \cite{San17}, the conjecture is proved for certain mapping classes of $\mathrm{Mod}(\Sigma_{0,n})$ where $n\geq 5,$ by relating the quantum representation to homological representations on ramified covers of the surface $\Sigma_{0,n}.$

The first pseudo-Anosov examples of the AMU conjecture in genus at least two were obtained by Marché and Santharoubane in \cite{MS21}. Those examples belong to $\mathrm{Mod}(\Sigma_{g,1})$ and are elements of the point-pushing subgroups coming from the Birman exact sequence. 

In \cite{DK19}, the author and Kalfgianni constructed many more pseudo-Anosov examples of the conjecture, belonging to $\mathrm{Mod}(\Sigma_{g,n})$ where either $n=2$ and $g\geq 3$ or $g\geq n \geq 2.$ The examples occured as monodromies of fibered links whose Turaev-Viro invariants grow exponentially. In a follow-up work \cite{DK22}, we extended this construction to provide pseudo-Anosov mapping classes in $\mathrm{Mod}(\Sigma_{g,1})$ for $g\geq 9,$ as monodromies of open book decompositions on $3$-manifolds whose Turaev-Viro invariants grow exponentially. 

We note that a common feature of all those examples is that they rely on asymptotic estimates of the $\rho_r(f)$ (and often, of their traces). Indeed, it is shown in \cite{DK19} that the AMU conjecture would follow from the Chen-Yang volume conjecture \cite{CY18}.

In this article, we use an entirely different strategy, replacing analysis by number theory as our main tool. Let $J_2(\Sigma_{g,n})$ be the Johnson kernel, that is, the subgroup of $\mathrm{Mod}(\Sigma_{g,n})$ generated by Dehn twists along separating cruves. A celebrated theorem of Johnson \cite{Joh85} characterizes this subgroup for $n\leq 1$ as the subgroup that acts trivially on $\pi_1(\Sigma_{g,n})/\Gamma^{3} (\pi_1(\Sigma_{g,n})),$ where $$\Gamma^{k}(G)=[G,[G,[\ldots,[G,G]]]]$$ ($G$ appears $k$ times). For $k\geq 1,$ and $n\leq 1,$  the $k$-th Johnson kernel was further defined to be the subgroup which acts trivially on $\pi_1(\Sigma_{g,n})/\Gamma^{k+1} (\pi_1(\Sigma_{g,n})).$ We prove:
\begin{thm}
	\label{thm:main} Let $\Sigma_{g,n}$ be a compact oriented surface of genus $g$ and $n$ boundary components. Let $f\in [J_2(\Sigma_{g,n}),J_2(\Sigma_{g,n}))]$ be non-trivial. Then for any  prime $p$ large enough, $\rho_p(f)$ has infinite order.
\end{thm}

The above theorem in particular provides the first pseudo-Anosov examples of the AMU conjecture for closed surfaces. Moreover, it is the first time that the AMU conjecture is proved for a natural subgroup of the mapping class group.

\begin{remark}
	\label{rk:pApart} An immediate consequence of Theorem \ref{thm:main} is that any non-trivial element of $[J_2(\Sigma_{g,n}),J_2(\Sigma_{g,n})]$ has a pseudo-Anosov part, since elements without a pseudo-Anosov part always have finite order images  by the representations $\rho_p.$ 
	
	While this fact is not surprising, the author was not able to locate a written proof of this in the litterature. 
	
	It should however be possible to even prove that any non-trivial element of $J_3(\Sigma_g)$ has a pseudo-Anosov part, using that Torelli groups are torsion-free (see \cite[Theorem 6.12]{FarbMargalit}) and showing that multitwists in the Torelli group have non-trivial image either by the first or second Johnson homomorphism.
\end{remark}

Theorem \ref{thm:main} also has the following corollary:
\begin{cor}
	\label{cor:main} Let $f\in \mathrm{Mod}(\Sigma_{g,n})$ be such that $f$ has a pseudo-Anosov part and that $f$ has finite order in $\mathrm{Mod}(\Sigma_{g,n})/[J_2(\Sigma_{g,n}),J_2(\Sigma_{g,n})].$ Then for any  prime $p$ large enough, $\rho_p(f)$ has infinite order. 
\end{cor}
The corollary follows from Theorem \ref{thm:main} by applying it to $f^d$ where $d$ is the order of $f$ in $\mathrm{Mod}(\Sigma_{g,n})/[J_2(\Sigma_{g,n}),J_2(\Sigma_{g,n})];$ note that $f^d$ is non-trivial as $f$ has a pseudo-Anosov part and hence infinite order. 


\begin{remark}\label{rk:torsionElem}
	  One may construct pseudo-Anosov elements of $\mathrm{Mod}(\Sigma_{g,n})$ that are covered by Corollary \ref{cor:main} but not by Theorem \ref{thm:main} as follows. Take $f$ which is either a non-trivial finite order element of $\mathrm{Mod}(\Sigma_{g,n}),$ or an element of $J_2(\Sigma_{g,n})$ representing non-trivial torsion in $H_1(J_2(\Sigma_{g,n}),\Z)$ (the latter exists by \cite{NSS22}, see also \cite{FM26}). From \cite[Theorem 1.3]{Wat20}, it is possible to find $\varphi \in [J_2(\Sigma_{g,n}),J_2(\Sigma_{g,n})]$ such that $f\varphi^k$ is a pseudo-Anosov element for large enough $k.$ Moreover, this element represents a non-trivial torsion element in $\mathrm{Mod}(\Sigma_{g,n})/[J_2(\Sigma_{g,n}),J_2(\Sigma_{g,n})].$
\end{remark}
Our techniques also yield the following:
\begin{thm}
	\label{thm:main2} Let $\Sigma_{g,n}$ be a compact oriented surface of genus $g$ and $n$ boundary components. Let $f\in J_2(\Sigma_{g,n})$ be non-trivial. Then for any  prime $p\geq 5$ large enough, $\rho_p(f)$ has either order $p$ or infinite order.
\end{thm}

We note that the above theorem is optimal, in the sense that separating Dehn twists belong to $J_2(\Sigma_{g,n})$ and are well-known to have images of order $p$ by $\rho_p.$

The author wonders if one could prove a version of Theorem \ref{thm:main} for the third term $J_3(\Sigma_{g})$ of the Johnson filtration, since, according to Remark \ref{rk:pApart}, any non-trivial element of $J_3(\Sigma_{g})$ has a pseudo-Anosov part.

The proofs of Theorem \ref{thm:main} and \ref{thm:main2} rely on a number-theoretic criterion (Propositions \ref{prop:infiniteOrder} and \ref{prop:infiniteOrder2})for a matrix with entries in a cyclotomic ring to have infinite order, which we prove in Section \ref{sec:infiniteOrder}. In addition, we make use of the integrality of the WRT representations $\rho_p$ and the structure of their restriction to $J_2(\Sigma_{g,n}),$ which was studied in \cite{GM07}. Finally, we make use of the asymptotic faithfulness of the WRT representations, due to Andersen \cite{And06} and Freedman, Walker and Wang \cite{FWW02}.

\textbf{Acknowledgements:} The author thanks Pierre Godfard, Julien March\'e, Gregor Masbaum and Ramanujan Santharoubane for their interest and feedback on a preliminary version of this paper. The IMB, host institution of the author, receives support from the EIPHI Graduate School (contract ANR-17-EURE-0002).

\section{A number-theoretic criterion for elements of $GL_d(\Z[\zeta_p])$ to be of infinite order}
\label{sec:infiniteOrder}
In this section, let $p$ be a prime, and let $\Z[\zeta_p]$ be the ring of cyclotomic integers, where $\zeta_p$ is a primitive $p$-th root of unity.
The goal of this section will be the proof of the following criterion for a matrix $A\in GL_d(\Z[\zeta_p])$ to have infinite order:
\begin{prop}
	\label{prop:infiniteOrder} Let $d \geq 1,$ let $p$ be a prime, and let $A\in GL_d(\Z[\zeta_p])$ be such that $A\neq I_d$ and $A=I_d \mod (\zeta_p-1)^2.$ Then $A$ has infinite order.
\end{prop}
The above criterion may be thought as a variation of the well-known lemma, due to Minkowski, that $GL_d(\Z/n\Z)$ is torsion-free for $n\geq 3$ (see \cite[Lemma 1]{Serre}).

To prove Proposition \ref{prop:infiniteOrder}, we will need a few elementary number theory facts. While much of the material of this section may be standard, we will recall all definitions and properties, for the sake of completeness.

We denote by $\N$ the set of positive integers, by $\Q_+$ and $\R_+$ the sets of positive rational and real numbers, and by $\overline{\Z}$ the ring of algebraic integers. We will denote by $\N^{\Q_+}$ the submonoid of $(\R_+,\times)$ generated by all elements of the form $n^q$ where $n\in \N$ and $q\in \Q_+.$
It follows from the fundamental theorem of arithmetics that any element  $x\in \N^{\Q_+}$ has a unique factorization
$$x=\underset{p \ \textrm{prime}}{\prod}p^{v_p(x)}$$
where $v_p(x)=0$ for all but finitely many primes $p$ and $v_p(x)\in \Q_+$ otherwise. 

\begin{remark}\label{rk:divisibility} For $x,y\in \N^{\Q_+},$ we say that $x$ divides $y$ if there exists $z\in \N^{\Q_+}$ such that $y=xz.$ It is immediate to see that $x$ divides $y$ if and only if $v_p(x)\leq v_p(y)$ for all primes $p.$
\end{remark}

For $z\in \overline{\Z}$ its degree $\deg(z)$ is the degree of the extension $\Q\subset \Q(z),$ the conjugates of $z$ are the roots of its minimal polynomial $\mu_z \in \Z[x].$ Equivalently, the set of conjugates of $z$ is the orbit of $z$ under $\mathrm{Gal}_{\overline{\Q}/\Q}.$  Finally we define its norm as follows:

\begin{definition}\label{def:norm} The (normalized) norm on $\overline{\Z}$ is the function
	$$N:\overline{\Z} \longrightarrow \N^{\Q_+} \cup \lbrace 0 \rbrace$$
	such that for any $z\in \overline{\Z},$ we have $$N(z)=\left(\underset{z_i \ \textrm{conjugate of} \ z}{\prod}|z_i| \right)^{\frac{1}{\deg(z)}}.$$
	\end{definition}

We note that the product of the conjugates is also sometimes refered to as the norm; our choice of normalization has the advantage that $N$ is multiplicative on $\overline{\Z}:$

\begin{prop}
	\label{prop:multNorm} For any $z\in \overline{\Z},$ and any finite extension $K$ of $\Q$ containing $z,$ we have
	$$N(z)=\left| \det(m_z) \right|^{\frac{1}{\dim_{\Q}(K)}},$$
	where $m_z: K\longrightarrow K$ is the $\Q$-linear operator of multiplication by $z.$ Moreover, for any $z,z'\in \overline{\Z},$ one has $N(zz')=N(z)N(z').$
\end{prop}
\begin{proof}
	The last claim follows from the first by setting $K=\Q(z,z'),$ since one has $m_{zz'}=m_zm_{z'}$ and the determinant is multiplicative.
	
	Let us prove the first claim. Let $d=\deg(z)$ and $\mu_z$ be the minimal polynomial of $z.$ Let also $u_1,\ldots,u_m$ be a basis of $K$ over $\Q(z).$ Then $$\mathcal{B}=\lbrace u_1,u_1z,\ldots,u_1z^{d-1},\ldots,u_m,\ldots,u_mz^{d-1}\rbrace$$ is a basis of $K$ over $\Q,$ and in this basis $m_z$ is block diagonal, with each block being the companion matrix $C_{\mu_z}$ of the monic polynomial $\mu_z.$ However
	$$|\det(C_{\mu_z})|=|\mu_z(0)|=\underset{z_i \ \textrm{conjugate of} \ z}{\prod}|z_i|.$$
	The claim follows.
\end{proof}

We  will now  compute the norm of some specific elements of cyclotomic rings. To start with, recall that the $n$-th cyclotomic polynomial is 
$$\phi_n(X)=\underset{k\leq n,\gcd(k,n)=1}{\prod}(X-e^{\frac{2ik\pi}{n}}).$$
It is well-known that $\phi_n\in \Z[X]$ is a monic irreducible polynomial, and is the minimal polynomial of any primitive $n$-th root of unity. The polynomials $\phi_n$ satisfy the recurrence relation:
$$X^n-1=\underset{d|n}{\prod}\phi_d(X).$$
 The degree of $\phi_n$ is given by $\varphi(n),$ the Euler totient function, which is multiplicative and satisfies $\varphi(p^k)=p^{k-1}(p-1)$ when $p$ is a prime.

\begin{lem}
 	\label{lemma:cyclotomNorm} Let $\zeta_n$ be a primitive $n$-th root of unity. Then $N(\zeta_n-1)=\phi_n(1)^{\frac{1}{\varphi(n)}}.$
\end{lem}

\begin{proof}
	Without loss of generality, assume $\zeta_n=e^{\frac{2i\pi}{n}}.$ The conjugates of $\zeta_n-1$ are all $\zeta_n^k-1$ for $k\leq n$ coprime to $n.$ In particular, the degree of $\zeta_n-1$ is $\varphi(n),$ and moreover:
	$$N(\zeta_n-1)=\left( \underset{k\leq n, \gcd(k,n)=1}{\prod}|1-\zeta_n^k|\right)^{\frac{1}{\varphi(n)}}=\phi_n(1)^{\frac{1}{\varphi(n)}}.$$
\end{proof}
 
The following lemma will make the computation of norms in Lemma \ref{lemma:cyclotomNorm} more explicit:
\begin{lem}
	\label{lemma:cyclotomicPolAtOne} Let $n \geq 2,$ then 
	$$\phi_n(1)=\begin{cases}
	p \ \textrm{if} \ n=p^k \ \textrm{where} \ p \ \textrm{is prime}
	\\ 1 \ \textrm{else}
	\end{cases}.$$
\end{lem}
\begin{proof}
	We prove the claim by induction on $n.$ The claim is clear for $n=2$ since $\phi_2(X)=X+1.$ Assume the claim true for all $m<n.$ We have
	$$\phi_n(X)=\frac{X^n-1}{(X-1) \underset{d|n, 1<d<n}{\prod} \phi_d(X)}.$$
	Since $\frac{X^n-1}{X-1}=1+X+\ldots +X^{n-1},$ evaluating at $1$ we get:
	$$\phi_n(1)=\frac{n}{\underset{d|n, 1<d<n}{\prod} \phi_d(X)}.$$
	If $n=p^k,$ by the induction hypothesis, the strict divisors of $n$ each contribute a factor $p$ to the denominator and we get $\phi_{p^k}(1)=p.$ If however $n$ is composite, the prime power divisors of $n$ are all strict divisors of $n$ and we get $\phi_n(1)=1$ by the induction hypothesis.
\end{proof}
We are now ready to prove Proposition \ref{prop:infiniteOrder}:

\begin{proof}[Proof of Proposition \ref{prop:infiniteOrder}]
	Let $d\geq 1,$ let $p$ be a prime and let $A\in GL_d(\Z[\zeta_p])$ such that $A\neq I_d$ but $A=I_d \mod (\zeta_p-1)^2.$
	
	Let us write $A=I_d +(\zeta_p-1)^2 M$ where $M\in M_d(\Z[\zeta_p]).$ Assume for a contradiction that $A$ has finite order $m\geq 2.$ Then $A$ has an eigenvalue $u_n$ which is a primitive $n$-th root for some $n\geq 2$ dividing $m.$ Therefore, $M$ admits $\frac{u_n-1}{(\zeta_p-1)^2}$ as an eigenvalue. However, since $M\in M_d(\Z[\zeta_p]),$ all its eigenvalues are algebraic integers. Therefore, $(\zeta_p-1)^2$ divides $(u_n-1)$ in $\overline{\Z},$ and thus by Proposition \ref{prop:multNorm}, we have that $N(\zeta_p-1)^2$ divides $N(u_n-1).$
	
	By Lemma \ref{lemma:cyclotomNorm} and \ref{lemma:cyclotomicPolAtOne}, we have that $N(\zeta_p-1)^2=p^{\frac{2}{p-1}}$ while $v_p(N(u_n-1))=0$ unless $n=p^k$ is a power of $p,$ in which case
	$$N(u_{p^k}-1)=p^{\frac{1}{p^{k-1}(p-1)}}.$$
	Therefore $N(\zeta_p-1)^2$ can not divide $N(u_n-1)$ and we have a contradiction. 
	
	Thus $A$ has infinite order.
\end{proof}

Finally, when we weaken slightly the hypothesis of Proposition \ref{prop:infiniteOrder}, we get the following:

\begin{prop}
	\label{prop:infiniteOrder2} Let $d\geq 1,$ let $p$ be a prime and let $A\in GL_d(\Z[\zeta_p])$ be such that $A\neq I_d$ but $A=I_d \mod (\zeta_p-1).$ Then $A$ either has order $p,$ or has infinite order.
\end{prop}

\begin{proof}
	As in the proof of Proposition \ref{prop:infiniteOrder}, set $A=I_d+(\zeta_p-1)M.$ If $A$ has finite order $m,$ then $M$ has an eigenvalue of the form $\frac{u_n-1}{\zeta_p-1}$ where $u_n$ is a primitive $n$-th root of unity for some $n\geq 2$ dividing $m.$ In fact, $n$ can be chosen to be a multiple of any prime power divisor of $m.$ By Lemma \ref{lemma:cyclotomNorm} and \ref{lemma:cyclotomicPolAtOne}, we get $N(\zeta_p-1)$ does not divide $N(u_n-1)$ unless $n=p,$ which means that $A$ has order exactly $p.$
	
	Therefore, $A$ either has order $p$ or infinite order.
\end{proof}

\begin{remark}
	\label{rk:projective} While we have stated Proposition \ref{prop:infiniteOrder} and \ref{prop:infiniteOrder2} for $A\in GL_d(\Z[\zeta_p]),$ they apply equally as well for $A\in PGL_d(\Z[\zeta_p]).$ One way to see this is to embed $PGL_d(\Z[\zeta_p])$ into $GL_{d^2}(\Z[\zeta_p])$ via the conjugation action on $M_d(\Z[\zeta_p]).$
\end{remark}

\section{The $h$-adic expansion and the AMU conjecture}
\label{sec:h-adic}
In this section, we let $\rho_p$ be the Witten-Reshetikhin-Turaev $\mathrm{SO}_3$-quantum representation of the mapping class group $\mathrm{Mod}(\Sigma_{g,n})$ of a compact oriented surface $\Sigma_{g,n}$ of genus $g$ and $n$ boundary components. We refer to \cite{BHMV} and \cite{GM07} for a definition. Following \cite{GM07}, the quantum representation $\rho_p$ is a projective representation:

$$\rho_p:\mathrm{Mod}(\Sigma_{g,n})\longrightarrow \mathrm{PAut}(\mathcal{S}_p(\Sigma_{g,n}))$$
where $\mathcal{S}_p(\Sigma_{g,n})$ is a free $\Z[\zeta_p]$-module of finite rank. The rank of $\mathcal{S}_p(\Sigma_{g,n})$ may be explicitely computed using the Verlinde formula \cite[Corollary 1.16]{BHMV}. We note that the module $\mathcal{S}_p(\Sigma_{g,n})$ breaks into direct summands $\mathcal{S}_p(\Sigma_{g,n},c),$ where $c:|\partial \Sigma_{g,n}|\longrightarrow \lbrace 0,2 \ldots, p-3 \rbrace,$ is a coloration of the boundary components of $\Sigma_{g,n}.$ Moreover, $\rho_p$ lifts to a linear representation of a central extension of $\mathrm{Mod}(\Sigma_{g,n}),$ with coefficients in a cyclotomic ring $\Z[\zeta_p]$ or $\Z[\zeta_{4p}]$ depending on $p \mod 4;$ however as a projective representation it may always be considered to have $\Z[\zeta_p]$-coefficients. 

We recall two important facts about these representations that will be used in our proof of Theorem \ref{thm:main}.

\begin{prop}
	\label{prop:J_2} \cite{GM07} For compact oriented surface $\Sigma_{g,n}$ with $n\leq 1,$ and for any $f\in J_2(\Sigma_{g,n}),$ one has $$\rho_p(f)=id_{\mathcal{S}_p(\Sigma_{g,n})} \mod (\zeta_p-1).$$
\end{prop}
\begin{remark}
	\label{rk:nLessThanOne} The restriction $n\leq 1$ may be explained as follows. By \cite[Remark 14.3]{GM07}, Dehn twists along separating curves in $\Sigma_{g,n}$ that have all the boundary components of $\Sigma_{g,n}$ on one side have images that are $=id_{S_p(\Sigma_{g,n})} \mod (\zeta_p-1).$ 
	
	One may notice from the description of the integral basis in \cite{GM07} that this is still true for separating curves that separate a $m$-holed sphere $\Sigma_{0,m}$ with $m\leq n+1$ on one side, but these twists are still not enough to generate $J_2(\Sigma_{g,n})$ if $n\geq 2$ and $g\geq 2.$ 
	
	Indeed, by \cite[Theorem A and B]{Boggi},  $J_2(\Sigma_{g,n})$ is generated by Dehn twists that separate a surface of the type $\Sigma_{2,1},\Sigma_{1,1},\Sigma_{1,2}$ or $\Sigma_{0,3}$ on one side; moreover, the latter is often a minimal normal generating set.
\end{remark}

The second ingredient will be the asymptotic faithfulness of $\mathrm{SO}_3$-WRT quantum representations which we state in the following way:

\begin{thm}\cite{And06}\cite{FWW02}
	
	\label{thm:asymptFaithful} For any compact oriented surface $\Sigma_{g,n},$ one has 
	$$\underset{p \geq 5, \ \textrm{prime}}{\bigcap} \ker \rho_p=Z(\mathrm{Mod}(\Sigma_{g,n})).$$
\end{thm}
In fact, the proofs of Theorem \ref{thm:asymptFaithful} even show that for a given $f\notin Z(\mathrm{Mod}(\Sigma_{g,n})),$ for any large enough prime $p\geq 5,$  one has $\rho_p(f)\neq id.$

Note that $Z(\mathrm{Mod}(\Sigma_{g,n}))$ is spanned by Dehn twists along the boundary components, unless if $(g,n)=(1,0)$ or $(2,0),$ in which case it is generated by the elliptic involution or hyperinvolution, which has order $2.$ In any case, the center never contains any element that has a pseudo-Anosov part.

Finally, we need the following lemma:

\begin{lem}
	\label{lem:fillingBoundComp} Let $S=\Sigma_{g,n}$ be a compact oriented surface and let $\widehat{S}$ be obtained from $S$ by gluing a one-holed torus on each boundary component of $S.$ Let $f\in \mathrm{Mod}(S)$ and let $\hat{f}$ be the mapping class obtained by extending $f$ to $\widehat{S}$ by the identity. Then for any odd $p \geq 5:$
	$$f\notin \ker \rho_p \Longleftrightarrow \hat{f} \notin \ker \rho_p.$$
\end{lem}
\begin{proof}
	We denote by $|\partial S|$ the set of boundary components of $S.$
	By the monoidality of the $RT_p$-TQFTs \cite{BHMV}, the vector space $RT_p(\widehat{S})$ may be decomposed as
	$$RT_p(\widehat{S})=\underset{c:|\partial S|\longrightarrow I_p }{\bigoplus} RT_p(S,c) \otimes \underset{x\in |\partial S|}{\bigotimes} RT_p(T_x,c(x)),$$
	where $T_x$ is the one-holed torus that is glued onto $x,$ and moreover 
	$$\rho_p(\hat{f})=\underset{c:|\partial S|\longrightarrow I_p }{\bigoplus} \rho_{p,c}(f) \otimes\underset{x\in |\partial S|}{\bigotimes} id_{RT_p(T_x,c(x))},$$
	where we recall that $I_p=\lbrace 0,2,\ldots,p-3\rbrace$ is the set of colors for the $\mathrm{SO}_3$-TQFT at level $p.$ The lemma follows.
\end{proof}
We can now prove our main Theorem:

\begin{proof}[Proof of Theorem \ref{thm:main}]
	Let $S=\Sigma_{g,n}$ be a compact oriented surface and let $f\in [J_2(S),J_2(S)].$ As in Lemma \ref{lem:fillingBoundComp}, let $\widehat{S}$ be the surface obtained by gluing a one-torus on each boundary component of $S,$ and let $\hat{f}$ be obtained by extending $f$ by the identity. By Lemma \ref{lem:fillingBoundComp}, if $\rho_p(f)$ has infinite order if and only if $\rho_p(\hat{f})$ has. Note that $J_2(S)\subset J_2(\widehat{S}).$ Indeed, an element $g\in J_2(S)$ is a composition of Dehn twists along separating curves in $S,$ therefore, $\hat{g}$ is a composition of Dehn twists along separating curves in $\widehat{S},$ since separating curves in $S$ stay separating in $\widehat{S}.$
	
	Similarly, if $f\in [J_2(S),J_2(S)],$ then $\hat{f}\in [J_2(\widehat{S}),J_2(\widehat{S})].$
	
	By Proposition \ref{prop:J_2}, $\rho_p(\hat{f})=id_{\mathcal{S}_p(\widehat{S})} \mod (\zeta_p-1)^2,$ for any odd prime $p\geq 5.$ However, by Theorem \ref{thm:asymptFaithful}, $\rho_p(\hat{f})\neq id_{\mathcal{S}_p(\widehat{S})}$ for any large enough prime $p.$ By Proposition \ref{prop:infiniteOrder} and Remark \ref{rk:projective}, we deduce that $\rho_p(\hat{f})$ has infinite order for any large enough prime $p,$ and therefore the same is true for $\rho_p(f).$
\end{proof}

The proof of Theorem \ref{thm:main2} is similar, but uses Proposition \ref{prop:infiniteOrder2} instead:

\begin{proof}[Proof of Theorem \ref{thm:main2}] Let $S=\Sigma_{g,n},$  let $f\in J_2(\Sigma_{g,n})$ and $\widehat{S}$ and $\hat{f}$ be defined as above.
	The orders of $\rho_p(f)$ and $\rho_p(\hat{f})$ are equal by Lemma \ref{lem:fillingBoundComp}, and $\rho_p(\hat{f})=id_{\mathcal{S}_p(\widehat{S})} \mod (\zeta_p-1)$ by Theorem \ref{prop:J_2}, but  $\rho_p(\hat{f})\neq id_{\mathcal{S}_p(\widehat{S})}$ for any large enough prime $p$ by Theorem \ref{thm:asymptFaithful}. We conclude using Proposition \ref{prop:infiniteOrder2} that $\rho_p(f)$ has either order $p$ or infinite order, for any large enough prime $p.$ 
\end{proof}

\bibliographystyle{hamsalpha}
\bibliography{biblio}
\end{document}